\documentclass[12pt]{article}
\usepackage{amssymb}
\usepackage{amsthm}
\usepackage{amsmath}
\usepackage{graphicx}
\usepackage{enumerate}
\usepackage{pict2e}
\usepackage{framed}

\DeclareMathOperator{\Max}{Max}
\DeclareMathOperator{\Min}{Min}

\newtheorem{theorem}{Theorem}[section]
\newtheorem{definition}[theorem]{Definition}
\newtheorem{lemma}[theorem]{Lemma}

\newtheorem{remark}[theorem]{Remark}
\newtheorem{example}[theorem]{Example}

\title{Operators $\Max L$ and $\Min U$ and duals of Boolean posets}
\author{Ivan~Chajda, Miroslav~Kola\v r\'ik and Helmut~L\"anger}
\date{}

\begin{document}

\footnotetext{Support of the research of the first author by the Czech Science Foundation (GA\v CR), project 25-20013L, and by IGA, project P\v rF~2025~008, is gratefully acknowledged. Support of the research of the third author by the Austrian Science Fund (FWF), project 10.55776/PIN5424624, is gratefully acknowledged.}

\maketitle
	
\begin{abstract}
When working with posets which are not necessarily lattices, one has a lack of lattice operations which causes problems in algebraic constructions. This is the reason why we use the operators $\Max L$ and $\Min U$ substituting infimum and supremum, respectively. We axiomatize these operators. Two more operators, namely the so-called symmetric difference and the Sheffer operator, are introduced and studied in complemented posets by using the operators $\Max L$ and $\Min U$. In Boolean algebras, the symmetric difference is used to construct its dual structure, the corresponding unitary Boolean ring. By generalizing this idea, we assign to each Boolean poset a so-called dual and prove that also, conversely, a Boolean poset can be derived from its dual.
\end{abstract}
	
	

\section{Introduction}

When working with lattices, we use the lattice operations $\wedge$ and $\vee$, which for couples of entries $x,y$ yield their infimum and supremum, respectively. When working with posets, the results of these operations are not defined in some cases. Instead of these operations, in posets one usually uses the operators of the lower cone $L(x,y)$ and the upper cone $U(x,y)$, respectively, see e.g.\ \cite{CR}, \cite N and \cite P. Similarly as in lattices, in posets we have $z\le x,y$ for all $z\in L(x,y)$ and $x,y\le w$ for all $w\in U(x,y)$, respectively. However, in the case when $\inf(x,y)$ or $\sup(x,y)$ exists, it may happen that $\inf(x,y)\ne L(x,y)$ and $\sup(x,y)\ne U(x,y)$. This is the reason why we prefer to use the operators $\Max L(x,y)$ and $\Min U(x,y)$ denoting the set of all maximal elements of the lower cone of $x$ and $y$ and the set of all minimal elements in the upper cone of $x$ and $y$, respectively. Then we have that in case $\inf(x,y)$ exists, it coincides with $\Max L(x,y)$, and in case $\sup(x,y)$ exists, it coincides with $\Min U(x,y)$, respectively. Hence, the operators $\Max L(x,y)$ and $\Min U(x,y)$ are closer to lattice operations than the operators $L(x,y)$ and $U(x,y)$ and this advantage is used in some constructions (cf.\ e.g.\ [2]) where congruences in posets are defined by using these operators. We axiomatize the operators $\Max L$ and $\Min U$ in bounded posets. Obviously, this axiomatization differs much from that of lattice operations since the operators in question do not share such nice properties as e.g.\ associativity.

When working with Boolean algebras, one often uses the term operation symmetric difference. In the second part of the paper we introduce the operator of symmetric difference and the Sheffer operator also for complemented posets. It is well-known that in Boolean algebras the Sheffer operation turns out to be very useful since by means of this single operation all the operations of the Boolean algebra can be generated. We axiomatize the Sheffer operator on complemented posets and we show how this Sheffer operator can induce the partial order relation, the complementation operation as well as the afore mentioned operators $\Max L$ and $\Min U$.

It is well-known that there is a natural one-to-one correspondence between Boolean algebras and unitary Boolean rings, see e.g.\ \cite B. This correspondence is realized by using the operation of symmetric difference which serves as the addition operation of the unitary Boolean ring. In the last part of the paper we establish a similar construction for Boolean posets by using the operator of symmetric difference.

\section{Basic concepts}

First we recall some more or less familiar concepts and then we formulate several results needed for our study.

Let $\mathbf P=(P,\le)$ be a poset, $a,b\in P$ and $A,B\subseteq P$. We say that $\mathbf P$ satisfies the {\em Ascending Chain Condition} ({\em ACC}, for short) or the {\em Descending Chain Condition} ({\em DCC}, for short) if $\mathbf P$ has no infinite ascending or descending chains, respectively. Let {\em $\Max A$} and {\em $\Min A$} denote the set of all maximal and minimal elements of $A$, respectively. If $\mathbf P$ satisfies the ACC or the DCC then for every non-empty subset $A$ of $P$ the set $\Max A$ or $\Min A$ is non-empty, respectively. We define
\begin{align*}
A\le B & \text{ if }x\le y\text{ for all }x\in A\text{ and all }y\in B, \\
  L(A) & :=\{x\in P\mid x\le A\}, \\
  U(A) & :=\{x\in P\mid A\le x\}.
\end{align*}
Here and in the following we identify singletons with their unique element. Instead of $L(\{a\})$, $L(\{a,b\})$, $L(A\cup\{a\})$, $L(A\cup B)$ and $L\big(U(A)\big)$ we simply write $L(a)$, $L(a,b)$, $L(A,a)$, $L(A,B)$ and $LU(A)$. Analogously we proceed in similar cases. 

Hence, by the binary operator $\Max L(x,y)$ and $\Min U(x,y)$ we mean the mapping assigning to each $(x,y)\in P^2$ the set of all maximal elements in the lower cone $L(x,y)$ of $x$ and $y$ and the set of all minimal elements in the upper cone $U(x,y)$ of $x$ and $y$, respectively. Of course, if the poset in question is a lattice then $\Max L(x,y)=x\wedge y$ and $\Min U(x,y)=x\vee y$. Thus also in the more general case of a poset, $\Max L(x,y)$  and $\Min U(x,y)$ should play the role of $x\wedge y$ and $x\vee y$, respectively.

The {\em poset} $\mathbf P$ is called {\em bounded} if it has a bottom element $0$ and a top element $1$. In this case we sometimes write $(P,\le,0,1)$.

Now we present our basic results.

\begin{lemma}\label{lem2}
Let $\mathbf P=(P,\le)$ be a poset satisfying the {\rm ACC} and the {\rm DCC} and let $A,B\subseteq P$. Then the following holds:
\begin{enumerate}[{\rm(i)}]
\item If $a\in A$ then there exists some $b\in\Min A$ and some $c\in\Max A$ with $b\le a\le c$,
\item $L(A)=L(B)$ if and only if $\Max L(A)=\Max L(B)$,
\item $U(A)=U(B)$ if and only if $\Min U(A)=\Min U(B)$,
\item $L(\Min A)=L(A)$ and $U(\Max A)=U(A)$.
\end{enumerate}
\end{lemma}

\begin{proof}
\
\begin{enumerate}[(i)]
\item Assume $a\in A$. Put $a_0:=a$. If $a_0\in\Min A$ then put $b:=a_0$. Otherwise there exists some $a_1\in A$ with $a_1<a_0$. If $a_1\in\Min A$ then put $b:=a_1$. Otherwise there exists some $a_2\in A$ with $a_2<a_1$. Since $\mathbf P$ satisfies the DCC this procedure must stop after a finite number of steps. Hence there exists some $k\ge0$ with $b=a_k\in\Min A$ and $b\le a$. The second assertion is dual.
\item If $L(A)=L(B)$ then $\Max L(A)=\Max L(B)$. Conversely, assume $\Max L(A)=\Max L(B)$ and $a\in L(A)$. According to (i) there exists some $b\in\Max L(A)$ with $a\le b$. Since $\Max L(A)=\Max L(B)$ we conclude $b\in\Max L(B)\subseteq L(B)$. This shows $L(A)\subseteq L(B)$. By symmetry we have $L(B)\subseteq L(A)$ and hence $L(A)=L(B)$,
\item This is dual to (ii).
\item Since $\Min A\subseteq A$ we have $L(A)\subseteq L(\Min A)$. Now assume $d\in L(\Min A)$ and $e\in A$. According to (i) there exists some $f\in\Min A$ with $f\le e$. Now $d\le f\le e$ and hence $d\le e$ showing $d\in L(A)$. This proves $L(\Min A)\subseteq L(A)$. Together we obtain $L(\Min A)=L(A)$. The second assertion is dual.
\end{enumerate}	
\end{proof}

It is evident that for any subset A of a bounded poset $(P,\le,0,1)$ the sets $L(A)$, $U(A)$, $\Max L(A)$ and $\Min U(A)$ are non-empty since $0\in L(A)$ and $1\in U(A)$.

Two elements $a$ and $b$ of a bounded poset $(P,\le,0,1)$ are called {\em complements} of each other if $L(a,b)=0$ and $U(a,b)=1$.

A unary operation $'$ on a poset $(P,\le)$ is called
\begin{itemize}
\item {\em antitone} if $x,y\in P$ and $x\le y$ together imply $y'\le x'$,
\item an {\em involution} if it satisfies the identity $(x')'\approx x$.
\end{itemize}
For a bounded poset $\mathbf P=(P,\le,{}',0,1)$ with an antitone involution $'$ we call this mapping $'$ a {\em complementation} and $\mathbf P$ a {\em complemented poset} if $x'$ is a complement of $x$ for every $x\in P$.

By a {\em binary operator} $\varphi$ on a set $P$ we mean a mapping from $P^2$ to $2^P$, i.e.\ to every couple of elements $x,y$ of $P$ there is assigned a subset $\varphi(x,y)$ of $P$. In the following such an operator is often extended from $P^2$ to $(2^P)^2$, i.e.\ to every couple of subsets $A,B$ of $P$ there is assigned a subset $\varphi(A,B)$ of $P$. In this case, elements of $P$ are identified with the corresponding one-element subsets of $P$.

\section{Axiomatization of the operators $\Max L$ and $\Min U$}

In this section we investigate the operators $\Max L$ and $\Min U$ on a poset. The aim of this section is to axiomatize these operators.

\begin{definition}\label{def1}
An {\em operator structure} is an ordered quintuple $(P,\sqcup,\sqcap,0,1)$ where $\sqcup$ and $\sqcap$ are binary operators of $P$ and $0,1\in P$ and the following conditions are satisfied:
\begin{enumerate}[{\rm(i)}]
\item $x\sqcup x\approx x$ and $x\sqcap x\approx x$,
\item $x\sqcup y\approx y\sqcup x$ and $x\sqcap y\approx y\sqcap x$,
\item $0\sqcup x\approx x$ and $x\sqcap1=x$,
\item $x\sqcup\big((x\sqcup y)\sqcup z\big)\approx0\sqcup\big((x\sqcup y)\sqcup z\big)$ and $x\sqcap\big((x\sqcap y)\sqcap z\big)\approx\big((x\sqcap y)\sqcap z\big)\sqcap1$,
\item $(x\sqcap y)\sqcup y\approx y$ and $x\sqcap(x\sqcup y)\approx x$,
\item $z\in x\sqcup y$ if and only if $x\sqcup z=y\sqcup z=z$ and if $u\in P$ and $x\sqcup u=y\sqcup u=u\sqcap z=u$ together imply $u=z$; $z\in x\sqcap y$ if and only if $z\sqcap x=z\sqcap y=z$ and if $u\in P$ and $u\sqcap x=u\sqcap y=z\sqcup u=u$ together imply $u=z$.
\end{enumerate}
\end{definition}

The conditions (i) -- (vi) are not independent. Observe that (i) and (iv) imply (iii) and that (iii) and (v) imply (i).

In what follows, we show that these operator structures can be used for an axiomatization of the operators $\Max L$ and $\Min U$.

\begin{theorem}\label{th2}
Let $\mathbf P=(P,\le,0,1)$ be a bounded poset satisfying the {\rm ACC} and the {\rm DCC}. Then $\mathbb A(\mathbf P):=(P,\Min U,\Max L,0,1)$ is an operator structure.	
\end{theorem}

\begin{proof}
Let $a,b,c\in P$. We check the conditions (i) -- (vi) of Definition~\ref{def1}.
\begin{enumerate}[(i)]
\item $\Min U(a,a)=a$ and $\Max L(a,a)=a$,
\item $\Min U(a,b)=\Min U(b,a)$ and $\Max L(a,b)=\Max L(b,a)$,
\item $\Min U(0,a)=a$ and $\Max L(a,1)=a$,
\item We have $a\le u\le v$ for all $u\in\Min U(a,b)$ and all $v\in\Min U\big(\Min U(a,b),c\big)$. Since $\Min U(a,b)\ne\emptyset$ according to the remark after Lemma~\ref{lem2}, we conclude $a\le v$ for all $v\in\Min U\big(\Min U(a,b),c\big)$. Because of $\Min U\big(\Min U(a,b),c\big)\ne\emptyset$ according to the remark after Lemma~\ref{lem2} we get
\begin{align*}
\Min U\Big(a,\Min U\big(\Min U(a,b),c\big)\Big) & =\Min U\Big(\Min U\big(\Min U(a,b),c\big)\Big)= \\
                                                & =\Min U\Big(0,\Min U\big(\Min U(a,b),c\big)\Big).
\end{align*}
The second identity follows by duality.
\item According to Lemma~\ref{lem2} we have
\[
\Min U\big(\Max L(a,b),b\big)=\Min U\big(L(a,b),b\big)=\Min U(b)\approx b.
\]
The second identity follows by duality.
\item We have $c\in\Min U(a,b)$ if and only if $c\in U(a,b)$ and if $c\ge d\in U(a,b)$ implies $c=d$.
\end{enumerate}
\end{proof}

Also conversely, we show how to obtain a bounded poset with the operators $\Max L$ and $\Min U$ from a given operator structure.

\begin{theorem}\label{th3}
Let $\mathbf A=(P,\sqcup,\sqcap,0,1)$ be an operator structure. Define a binary relation $\le$ on $P$ by $x\le y$ if $x\sqcup y=y$ {\rm(}$x,y\in P${\rm)}. Then $\mathbb P(\mathbf A):=(P,\le,0,1)$ is a bounded poset satisfying the following conditions:
\begin{enumerate}[{\rm(a)}]
\item $x\le y$ if and only if $x\sqcap y=x$,
\item $x\sqcup y\approx\Min U(x,y)$,
\item $x\sqcap y\approx\Max L(x,y)$.
\end{enumerate}
\end{theorem}

\begin{proof}
Let $a,b,c\in P$. We use the conditions of Definition~\ref{def1}. According to (i), $\le$ is reflexive. If $a\le b$ and $b\le a$ then $a=b\sqcup a=a\sqcup b=b$ by (ii) proving antisymmetry of $\le$. If $a\le b$ and $b\le c$ then
\[
a\sqcup c=a\sqcup(b\sqcup c)=a\sqcup\big((a\sqcup b)\sqcup c)=0\sqcup\big((a\sqcup b)\sqcup c\big)=0\sqcup(b\sqcup c)=0\sqcup c=c
\]
according to (iii) and (iv) proving transitivity of $\le$. Altogether, $(P,\le)$ is a poset.
Because of (iii), $0$ is the bottom element of $(P,\le)$.
\begin{enumerate}[(a)]
\item From (v) we conclude that $a\sqcup b=b$ if and only if $a\sqcap b=a$. Due to (iii), $1$ is the top element of $(P,\le)$.
\item and (c) follow from (vi).
\end{enumerate}
\end{proof}

That the correspondence described in Theorems~\ref{th2} and \ref{th3} is nearly one-to-one is proved in the next theorem.

\begin{theorem}
\
\begin{enumerate}[{\rm(i)}]
\item If $\mathbf P=(P,\le,0,1)$ is a bounded poset satisfying the {\rm ACC} and the {\rm DCC} then $\mathbb P\big(\mathbb A(\mathbf P)\big)=\mathbf P$.
\item If $\mathbf A=(P,\sqcup,\sqcap,0,1)$ is a finite operator structure and $\mathbb A\big(\mathbb P(\mathbf A)\big)=(P,\sqcup_1,\sqcap_1,0,1)$ then $x\sqcup_1y=x\sqcup y$ and $x\sqcap_1y=x\sqcap y$ for all $x,y\in P$.
\end{enumerate}
\end{theorem}

\begin{proof}
\
\begin{enumerate}[(i)]
\item If $\mathbf P=(P,\le,0,1)$ is a bounded poset satisfying the {\rm ACC} and the {\rm DCC}, $\mathbb A(\mathbf P)=(P,\Min U,\Max L,0,1)$, $\mathbb P\big(\mathbb A(\mathbf P)\big)=(P,\sqsubseteq,0,1)$ and $a,b\in P$ then the following are equivalent: $a\sqsubseteq b$, $\Min U(a,b)=b$, $a\le b$.	
\item If $\mathbf A=(P,\sqcup,\sqcap,0,1)$ is a finite operator structure, $\mathbb P(\mathbf A)=(P,\le,0,1)$, $\mathbb A\big(\mathbb P(\mathbf A)\big)=(P,\sqcup_1,\sqcap_1,0,1)$ and $a,b\in P$ then
\begin{align*}
a\sqcup_1b & =\Min U(a,b)=a\sqcup b, \\
a\sqcap_1b & =\Max L(a,b)=a\sqcup b.
\end{align*}
\end{enumerate}
\end{proof}	

\begin{remark}
Contrary to the case of the lattice operations $\wedge$ and $\vee$, the binary operators $\Max L$ and $\Min U$ need not be associative, i.e.\ they need not satisfy the identities
\begin{align*}
\Max L\big(\Max L(x,y),z\big) & \approx\Max L\big(x,\Max L(y,z)\big), \\
\Min U\big(\Min U(x,y),z\big) & \approx\Min U\big(x,\Min U(y,z)\big).
\end{align*}
Consider the poset depicted in Fig.~1
	
\vspace*{-4mm}
	
\begin{center}
\setlength{\unitlength}{7mm}
\begin{picture}(6,6)
\put(2,1){\circle*{.3}}
\put(1,3){\circle*{.3}}
\put(3,3){\circle*{.3}}
\put(1,5){\circle*{.3}}
\put(3,5){\circle*{.3}}
\put(5,5){\circle*{.3}}
\put(1,3){\line(0,1)2}
\put(1,3){\line(1,1)2}
\put(1,3){\line(1,-2)1}
\put(3,3){\line(-1,-2)1}
\put(3,3){\line(-1,1)2}
\put(3,3){\line(0,1)2}
\put(3,3){\line(1,1)2}
\put(1.85,.3){$0$}
\put(.35,2.85){$a$}
\put(3.35,2.85){$b$}
\put(.35,4.85){$c$}
\put(3.35,4.85){$d$}
\put(5.35,4.85){$e$}
\put(1,-.75){{\rm Figure~1. Poset}}
\end{picture}
\end{center}
	
\vspace*{4mm}
	
Then
\[
\Max L\big(\Max L(c,d),e\big)=\Max L(a,b,e)=0\ne b=\Max L(c,b)=\Max L\big(c,\Max L(d,e)\big).
\]
Dually one can show that the operator $\Min U$ is not associative.
\end{remark}

\section{Symmetric difference and the Sheffer operator}

We continue our investigations by considering another operator on a complemented poset, namely the so-called {\em symmetric difference} which we will introduce by means of the operators $\Max L$ and $\Min U$.

Let $(P,\le,{}',0,1)$ be a complemented poset. For every subset $A$ of $P$ put $A':=\{a'\mid a\in A\}$. In a complemented lattice $(L,\vee,\wedge,{}',0,1)$ the so-called {\em symmetric difference} $+$ is usually defined as follows (see e.g.\ \cite B):
\[
x+y:=(x'\wedge y)\vee(x\wedge y')
\]
for all $x,y\in L$. A similar binary operator can be introduced in the complemented poset $(P,\le,{}',0,1)$ by
\[
x+y:=\Min U\big(\Max L(x',y),\Max L(x,y')\big)
\]
for all $x,y\in P$. We extend this definition from elements $x,y$ of $P$ to all subsets $A,B$ of $P$ by defining
\[
A+B:=\Min U\big(\Max L(A',B),\Max L(A,B')\big).
\]
Although the result of $x+y$ need not be an element of $P$, but may be a subset of $P$, the corresponding operator on complemented posets shares some familiarly known properties of the usual symmetric difference in lattices, in particular we have
\begin{align*}
& x+x\approx0, x+y\approx y+x, x+0\approx x, x+1\approx x', (x+1)+1\approx x, x+x'\approx1, \\
& x+y\approx x'+y', x+y'\approx x'+y.
\end{align*}
Recall that a poset $(P,\le)$ is called {\em distributive} if it satisfies one of the following equivalent conditions (see \cite N and \cite P):
\begin{align*}
L\big(U(x,y),z\big) & \approx LU\big(L(x,z),L(y,z)\big), \\	
U\big(L(x,y),z\big) & \approx UL\big(U(x,z),U(y,z)\big), \\
L\big(U(x,y),A\big) & \approx LU\big(L(x,A),L(y,A)\big), \\	
U\big(L(x,y),A\big) & \approx UL\big(U(x,A),U(y,A)\big)
\end{align*}
where $A$ runs through all subsets of $P$.

We demonstrate the concept of symmetric difference by the following example.

\begin{example}
Consider the poset visualized in Fig.~2

\vspace*{-4mm}

\begin{center}
\setlength{\unitlength}{7mm}
\begin{picture}(8,8)
\put(4,1){\circle*{.3}}
\put(1,3){\circle*{.3}}
\put(3,3){\circle*{.3}}
\put(5,3){\circle*{.3}}
\put(7,3){\circle*{.3}}
\put(1,5){\circle*{.3}}
\put(3,5){\circle*{.3}}
\put(5,5){\circle*{.3}}
\put(7,5){\circle*{.3}}
\put(4,7){\circle*{.3}}
\put(4,1){\line(-3,2)3}
\put(4,1){\line(-1,2)1}
\put(4,1){\line(1,2)1}
\put(4,1){\line(3,2)3}
\put(4,7){\line(-3,-2)3}
\put(4,7){\line(-1,-2)1}
\put(4,7){\line(1,-2)1}
\put(4,7){\line(3,-2)3}
\put(1,3){\line(0,1)2}
\put(1,3){\line(1,1)2}
\put(1,3){\line(2,1)4}
\put(3,3){\line(-1,1)2}
\put(3,3){\line(2,1)4}
\put(5,3){\line(-2,1)4}
\put(5,3){\line(1,1)2}
\put(7,3){\line(-2,1)4}
\put(7,3){\line(-1,1)2}
\put(7,3){\line(0,1)2}
\put(3.85,.3){$0$}
\put(.35,2.85){$a$}
\put(2.35,2.85){$b$}
\put(5.4,2.85){$c$}
\put(7.4,2.85){$d$}
\put(.35,4.85){$d'$}
\put(2.35,4.85){$c'$}
\put(5.4,4.85){$b'$}
\put(7.4,4.85){$a'$}
\put(3.85,7.4){$1$}
\put(0,-.75){{\rm Figure~2. Complemented poset}}
\end{picture}
\end{center}

\vspace*{4mm}

This poset is complemented, but not distributive since
\[
L\big(U(a,b),c\big)=L(d',1,c)=L(c)\ne0=LU(0,0)=LU\big(L(a,c),L(b,c)\big).
\]
The table for the symmetric difference $+$ is as follows:
\newpage
\[
\begin{array}{l|llllllllll}
+      & 0  & a    & b    & c    & d    & a'   & b'   & c'   & d'   & 1 \\
\hline
0      & 0  & a    & b    & c    & d    & a'   & b'   & c'   & d'   & 1 \\
a      & a  & 0    & d'   & d'   & b'c' & 1    & d    & d    & a'd' & a' \\
b      & b  & d'   & 0    & 0    & a'   & d    & 1    & b'c' & a    & b' \\
c      & c  & d'   & 0    & 0    & a'   & d    & b'c' & 1    & a    & c' \\
d      & d  & b'c' & a'   & a'   & 0    & a'd' & a    & a    & 1    & d' \\
a'     & a' & 1    & d    & d    & a'd' & 0    & d'   & d'   & b'c' & a \\
b'     & b' & d    & 1    & b'c' & a    & d'   & 0    & 0    & a'   & b \\
c'     & c' & d    & b'c' & 1    & a    & d'   & 0    & 0    & a'   & c \\
d'     & d' & a'd' & a    & a    & 1    & b'c' & a'   & a'   & 0    & d \\
1      & 1  & a'   & b'   & c'   & d'   & a    & b    & c    & d    & 0
\end{array}
\]
\hspace*{71mm} {\rm Table~1} \\

\vspace*{-3mm}

Here and in the following we write $b'c'$ instead of $\{b',c'\}$ and so on.
\end{example}

Recall that a {\em Boolean poset} is a distributive complemented poset.

From now on we concentrate in Boolean posets.

Distributive posets were investigated in \cite{CR}, \cite{LR} and \cite P. For Boolean posets the reader is referred to \cite N.

Examples of Boolean posets that are not lattices are as follows:

\vspace*{-4mm}

\begin{center}
\setlength{\unitlength}{7mm}
\begin{picture}(8,8)
\put(4,1){\circle*{.3}}
\put(1,3){\circle*{.3}}
\put(3,3){\circle*{.3}}
\put(5,3){\circle*{.3}}
\put(7,3){\circle*{.3}}
\put(1,5){\circle*{.3}}
\put(3,5){\circle*{.3}}
\put(5,5){\circle*{.3}}
\put(7,5){\circle*{.3}}
\put(4,7){\circle*{.3}}
\put(4,1){\line(-3,2)3}
\put(4,1){\line(-1,2)1}
\put(4,1){\line(1,2)1}
\put(4,1){\line(3,2)3}
\put(4,7){\line(-3,-2)3}
\put(4,7){\line(-1,-2)1}
\put(4,7){\line(1,-2)1}
\put(4,7){\line(3,-2)3}
\put(1,3){\line(0,1)2}
\put(1,3){\line(1,1)2}
\put(1,3){\line(2,1)4}
\put(3,3){\line(-1,1)2}
\put(3,3){\line(0,1)2}
\put(3,3){\line(2,1)4}
\put(5,3){\line(-2,1)4}
\put(5,3){\line(0,1)2}
\put(5,3){\line(1,1)2}
\put(7,3){\line(-2,1)4}
\put(7,3){\line(-1,1)2}
\put(7,3){\line(0,1)2}
\put(3.85,.3){$0$}
\put(.35,2.85){$a$}
\put(2.35,2.85){$b$}
\put(5.4,2.85){$c$}
\put(7.4,2.85){$d$}
\put(.35,4.85){$d'$}
\put(2.35,4.85){$c'$}
\put(5.4,4.85){$b'$}
\put(7.4,4.85){$a'$}
\put(3.85,7.4){$1$}
\put(.9,-.75){{\rm Figure~3. Boolean poset}}
\end{picture}
\end{center}

\vspace*{3mm}

The table for the symmetric difference is as follows:
\newpage
\[
\begin{array}{l|llllllllll}
	+      & 0  & a    & b    & c    & d    & a'   & b'   & c'   & d'   & 1 \\
	\hline
	0      & 0  & a    & b    & c    & d    & a'   & b'   & c'   & d'   & 1 \\
	a      & a  & 0    & c'd' & b'd' & b'c' & 1    & a'b' & a'c' & a'd' & a' \\
	b      & b  & c'd' & 0    & a'd' & a'c' & a'b' & 1    & b'c' & b'd' & b' \\
	c      & c  & b'd' & a'd' & 0    & a'b' & a'c' & b'c' & 1    & c'd' & c' \\
	d      & d  & b'c' & a'c' & a'b' & 0    & a'd' & b'd' & c'd' & 1    & d' \\
	a'     & a' & 1    & a'b' & a'c' & a'd' & 0    & c'd' & b'd' & b'c' & a \\
	b'     & b' & a'b' & 1    & b'c' & b'd' & c'd' & 0    & a'd' & a'c' & b \\
	c'     & c' & a'c' & b'c' & 1    & c'd' & b'd' & a'd' & 0    & a'b' & c \\
	d'     & d' & a'd' & b'd' & c'd' & 1    & b'c' & a'c' & a'b' & 0    & d \\
	1      & 1  & a'   & b'   & c'   & d'   & a    & b    & c    & d    & 0
\end{array}
\]
\hspace*{71mm} {\rm Table~2} \\

\vspace*{-3mm}

From Table~2 one can see that for $x,y\notin\{0,1\}$, $x\ne y$ and $x\ne y'$ we have $|x+y|>1$. That this does not hold in every Boolean poset is demonstrated by the next example.
\begin{center}
\setlength{\unitlength}{7mm}
\begin{picture}(8,10)
\put(4,1){\circle*{.3}}
\put(1,3){\circle*{.3}}
\put(3,3){\circle*{.3}}
\put(5,3){\circle*{.3}}
\put(7,3){\circle*{.3}}
\put(1,7){\circle*{.3}}
\put(3,7){\circle*{.3}}
\put(5,7){\circle*{.3}}
\put(7,7){\circle*{.3}}
\put(4,9){\circle*{.3}}
\put(1,5){\circle*{.3}}
\put(7,5){\circle*{.3}}
\put(4,1){\line(-3,2)3}
\put(4,1){\line(-1,2)1}
\put(4,1){\line(1,2)1}
\put(4,1){\line(3,2)3}
\put(4,9){\line(-3,-2)3}
\put(4,9){\line(-1,-2)1}
\put(4,9){\line(1,-2)1}
\put(4,9){\line(3,-2)3}
\put(1,3){\line(0,1)4}
\put(1,3){\line(1,1)4}
\put(3,3){\line(-1,1)2}
\put(3,3){\line(1,1)4}
\put(5,3){\line(-1,1)4}
\put(5,3){\line(1,1)2}
\put(7,3){\line(-1,1)4}
\put(7,3){\line(0,1)4}
\put(1,5){\line(1,1)2}
\put(7,5){\line(-1,1)2}
\put(3.85,.3){$0$}
\put(.35,2.85){$a$}
\put(2.35,2.85){$b$}
\put(5.4,2.85){$c$}
\put(7.4,2.85){$d$}
\put(.35,4.85){$e$}
\put(7.4,4.85){$e'$}
\put(.35,6.85){$d'$}
\put(2.35,6.85){$c'$}
\put(5.4,6.85){$b'$}
\put(7.4,6.85){$a'$}
\put(3.85,9.4){$1$}
\put(.85,-.75){{\rm Figure~4. Boolean poset}}
\end{picture}
\end{center}

\vspace*{4mm}

The table for the symmetric difference is as follows:
\[
\begin{array}{l|llllllllllll}
	+      & 0  & a    & b    & c    & d    & e  & a'   & b'   & c'   & d'   & e' & 1 \\
	\hline
	0      & 0  & a    & b    & c    & d    & e  & a'   & b'   & c'   & d'   & e' & 1 \\
	a      & a  & 0    & e    & b'd' & b'c' & b  & 1    & e'   & a'c' & a'd' & b' & a' \\
	b      & b  & e    & 0    & a'd' & a'c' & a  & e'   & 1    & b'c' & b'd' & a' & b' \\
	c      & c  & b'd' & a'd' & 0    & e'   & d' & a'c' & b'c' & 1    & e    & d  & c'\\
	d      & d  & b'c' & a'c' & e'   & 0    & c' & a'd' & b'd' & e    & 1    & c  & d' \\
    e      & e  & b    & a    & d'   & c'   & 0  & b'   & a'   & d    & c    & 1  & e' \\
	a'     & a' & 1    & e'   & a'c' & a'd' & b' & 0    & e    & b'd' & b'c' & b  & a \\
	b'     & b' & e'   & 1    & b'c' & b'd' & a' & e    & 0    & a'd' & a'c' & a  & b \\
	c'     & c' & a'c' & b'c' & 1    & e    & d  & b'd' & a'd' & 0    & e'   & d' & c \\
	d'     & d' & a'd' & b'd' & e    & 1    & c  & b'c' & a'c' & e'   & 0    & c' & d \\
	e'     & e' & b'   & a'   & d    & c    & 1  & b    & a    & d'   & c'   & 0  & e \\
	1      & 1  & a'   & b'   & c'   & d'   & e' & a    & b    & c    & d    & e  & 0
\end{array}
\]
\hspace*{71mm} {\rm Table~3} \\

\vspace*{-3mm}

\begin{remark}
The symmetric difference in Boolean algebras is associative, i.e.\ it satisfies the identity $(x+y)+z\approx x+(y+z)$. Unfortunately, this is not the case in Boolean posets. E.g., from the previous table we obtain
\begin{align*}
(a+b)+c & =e+c=d'\ne\{a',d'\}=\Min U(b,c,0)= \\
        & =\Min U\big(\Max L(a',d'),\Max L(a,d)\big)=a+\{a',d'\}=a+(b+c).
\end{align*}
However, the example of Figure~3 satisfies the so-called {\em weak distributivity}, i.e.
\[
\Min U\big((x+1)y'\big)\approx xy'+y'.
\]
\end{remark}

It is elementary to check that in Boolean algebras $(B,\vee,\wedge,{}',0,1)$ we have
\[
x+y\approx(x\vee y)\wedge(x'\vee y').
\]
This can be easily proved by using distributivity. We show that a similar result holds also for Boolean posets.

\begin{theorem}\label{th1}
Let $(P,\le,{}',0,1)$ be a Boolean poset. Then
\[
x+y\approx\Min U\Big(\Max L\big(\Min U(x,y),\Min U(x',y')\big)\Big).
\]
\end{theorem}

\begin{proof}
We compute
\begin{align*}
x+y & \approx\Min U\big(\Max L(x',y),\Max L(x,y')\big)\approx\Min U\big(L(x',y),L(x,y')\big)\approx \\
    & \approx\Min UL\Big(U\big(x',L(x,y')\big),U\big(y,L(x,y')\big)\Big) \\
    & \approx\Min UL\Big(UL\big(U(x',x),U(x',y')\big),UL\big(U(y,x),U(y,y')\big)\Big)\approx \\
    & \approx\Min UL\big(ULU(x',y'),ULU(y,x)\big)\approx\Min UL\big(U(x,y),U(x',y')\big)\approx \\
    & \approx\Min U\Big(\Max L\big(\Min U(x,y),\Min U(x',y')\big)\Big).
\end{align*}
\end{proof}

It is worth noticing that this result does not hold in complemented posets that are not distributive, e.g.\ in the complemented poset depicted in Figure~2 we have $b+c=0$, but
\begin{align*}
\Min U\Big(\Max L\big(\Min U(b,c),\Min U(b',c')\big)\Big) & =\Min U\big(\Max L(a',d',1)\big)=\Min U(b,c)= \\
                                                          & =\{a',d'\}.
\end{align*}
It is almost evident from Theorem~\ref{th1} that in a Boolean poset $(P,\le,{}',0,1)$ for elements $a,b$ of $P$ with $a\le b$ we have
\[
a+b=\Min U\big(\Max L(a',b)\big)=\Min UL(a',b).
\]	

Using a more restrictive property than distributivity, we can extend the result from Theorem~\ref{th1} from elements to subsets.

\begin{theorem}
Let $(P,\le,{}',0,1)$ be a complemented poset satisfying the identity
\[
U\big(L(A,B),C\big)\approx UL\big(U(A,C),U(B,C)\big)
\]
where $A,B,C$ run through all subsets of $P$, and let $A,B$ be non-empty subsets of $P$. Then
\[
A+B\approx\Min U\Big(\Max L\big(\Min U(A,B),\Min U(A',B')\big)\Big).
\]
\end{theorem}

\begin{proof}
We have
\begin{align*}	
A+B & \approx\Min U\big(L(A',B),L(A,B')\big)\approx\Min UL\Big(U\big(A',L(A,B')\big),U\big(B,L(A,B')\big)\Big)\approx \\
    & \approx\Min UL\Big(UL\big(U(A',A),U(A',B')\big),UL\big(U(B,A),U(B,B')\big)\Big)\approx \\
    & \approx\Min UL\big(ULU(A',B'),ULU(B,A)\big)\approx\Min UL\big(U(A,B),U(A',B')\big)\approx \\
    & \approx\Min U\Big(\Max L\big(\Min U(A,B),\Min U(A',B')\big)\Big).
\end{align*}
\end{proof}

As mentioned in the introduction, an important term operation in Boolean algebras is the so-called {\em Sheffer operation} (alias {\em Sheffer stroke}). It is usually denoted by $|$ and defined by $x|y:=x'\wedge y'$. We extend this concept to complemented posets.

\begin{definition}\label{def2}
A {\em Sheffer structure} is an ordered pair $(P,|)$ where $P$ is a non-empty set and $|$ is a binary operator on $P$ satisfying the following conditions {\rm(}we abbreviate $x|x$ by $x'$ for $x\in P$ and for subsets $A$ of $P$, $A|A$ by $A'${\rm)}:
\begin{enumerate}[{\rm(i)}]
\item $x|y\approx y|x$,
\item $(x|x)|(x|x)\approx x$,
\item $(x|y)|(x|x)\approx x$,
\item $\big(x'|(y'|z')'\big)'|z'\approx\big(x'|(y'|z')'\big)|(x|x')$,
\item $x|x'\in P$ and $x|x'\approx y|y'$,
\item $(x|x')'|x\approx x|x'$,
\item $z\in x|y$ if and only if $z'|x=z'|y=z$ and if $u'|x=u'|y=u$ and $z'|u'=z$ together imply $z=u$.
\end{enumerate}	
\end{definition}

The natural correspondence between a complemented poset and a Sheffer structure is described by the next two theorems.

\begin{theorem}
Let $(P,\le,{}',0,1)$ be a complemented poset satisfying the {\rm ACC} and the {\rm DCC} and define a binary operator $|$ on $P$ by $x|y:=\Max L(x',y')$ for all $x,y\in P$. Then $(P,|)$ is a Sheffer structure.
\end{theorem}

\begin{proof}
For subsets $A,B$ of $P$ we define $A|B:=\Max L(A',B')$. Let $a,b,c\in P$.
\begin{enumerate}[(i)]
\item $a|b=\Max L(a',b')=\Max L(b',a')=b|a$,
\item $a|a=\Max L(a',a')=a'$ and therefore the unary operator $'$ on the Sheffer structure restricted to $P$ coincides with the unary operation $'$ on the complemented poset,
\item $(a|b)|a'=\Max L\big(\Min U(a,b),a\big)=\Max L(a)=a$,
\item We have $a'|b'=\Max L(a,b)$ and $v\le u\le c$ for all $u\in\Max L(b,c)$ and all $v\in \Max L\big(a,\Max L(b,c)\big)$. Since $\Max L(b,c)\ne\emptyset$ according to the remark after Lemma~\ref{lem2}, we conclude $v\le c$ for all $v\in\Max L\big(a,\Max L(b,c)\big)$. Because of $\Max L\big(a,\Max L(b,c)\big)\ne\emptyset$ and due to the remark after Lemma~\ref{lem2} we obtain
\begin{align*}
\big(a'|(b'|c')'\big)'|c' & =\Max L\Big(\Max L\big(a,\Max L(b,c)\big),c\Big)= \\
                          & =\Max L\Big(\Max L\big(a,\Max L(b,c)\big)\Big)= \\
                          & =\Max L\Big(\Max L\big(a,\Max L(b,c)\big),\Max L(a',a)\Big)=\big(a'|(b'|c')'\big)'|(a|a').
\end{align*}
\item $a|a'=\Max L(a',a)=0=b|b'$,
\item $(a|a')'|a=0'|a=\Max L(0,a')=0=a|a'$,
\item It is easy to see that $a\le b$ if and only if $\Max L(a,b)=a$. Hence $c\in a|b$ if and only if $\Max L(c,a')=\Max L(c,b')=c$ and if $\Max L(d,a')=\Max L(d,b')=d$ and $\Max L(c,d)=c$ together imply $c=d$.
\end{enumerate}
\end{proof}

Now we show how a complemented poset together with the operators $\Max L$ and $\Min U$ can be received back from a Sheffer structure.

\begin{theorem}
Let $(P,|)$ be a Sheffer structure, put $x':=x|x$ for all $x\in P$, define a binary relation $\le$ on $P$ by $x\le y$ if $x'|y'=x$ and put $0:=x|x'$ and $1:=0'$ {\rm(}$x,y\in P${\rm)}. Then $\mathbf P:=(P,\le,{}',0,1)$ is a complemented poset, $x|y\approx\Max L(x',y')$ and $(x|y)'\approx\Min U(x,y)$.
\end{theorem}

\begin{proof}
We define $A':=A|A$ for all subsets $A$ of $P$. Let $a,b,c\in P$. We use the conditions of Definition~\ref{def2}. Observe that because of (v), $0$ is well-defined. By (ii), $'$ is an involution. Because of (ii), $\le$ is reflexive. If $a\le b$ and $b\le a$ then $a'|b'=a$ and $b'|a'=b$ and hence $a=b$ by (i). If $a\le b$ and $b\le c$ then $a'|b'=a$ and $b'|c'=b$ and hence by (i), (iii) and (iv)
\begin{align*} a'|c' & =(a'|b')'|c'=\big(a'|(b'|c')'\big)'|c'=\big(a'|(b'|c')'\big)'|(a|a')=(a'|b')'|(a|a')=a'|(a|a')= \\
	                 & =(a|a')|a'=a,
\end{align*}
i.e. $a\le c$. Therefore $(P,\le)$ is a poset. If $a\le b$ then $a'|b'=a$ and hence
\[
(b')'|(a')'=a|b=(a'|b')|b=(b'|a')|(b')'=b'
\]
according to (i) -- (iii), i.e.\ $b'\le a'$. Therefore, $'$ is antitone. Because of (i), (ii) and (vi) we have
\[
0'|a'=(a|a')'|a'=\big(a'|(a')'\big)|a'=a'|(a')'=a|a'=0,
\]
i.e.\ $0\le a$. Since $'$ is an antitone involution on $(P,\le)$ we conclude $a\le1$. This shows that $(P,\le,0,1)$ is a bounded poset. From (vii) we obtain $a|b=\Max L(a',b')$. Since $'$ is an antitone involution on $(P,\le)$ we obtain $(a|b)'=\Min U(a,b)$. Now $\Max L(a,a')=\Max L\big(a',(a')'\big)=a|a'=0$ and hence $L(a,a')=0$. Since $'$ is an antitone involution on $(P,\le)$ we get $U(a,a')=1$. This shows that $\mathbf P$ is a complemented poset.
\end{proof}

\section{Duals of Boolean posets}

It was established already by G.~Boole and formally described by G.~Birkhoff \cite B that Boolean algebras and unitary Boolean rings are in a natural one-to-one correspondence, in other words, unitary Boolean rings are structures being dual to Boolean algebras. One may ask whether there exists also a structure being dual to a Boolean poset. We construct such a structure by using the operators mentioned in the previous sections and we show that one can get back a Boolean poset from its dual.

We start with the following definition.

\begin{definition}\label{def3}
A {\em dual of a Boolean poset} is a structure $(P,+,\cdot,0,1)$ where $+$ and $\cdot$ are binary ``operators'' on $P$, i.e.\ they are defined also for subsets of $P$, and where $0,1\in P$ and the following conditions for all $x,y,z\in P$ and all $A,B\subseteq P$ are satisfied:
\begin{enumerate}[{\rm(i)}]
\item $x\cdot x=x$, $x\cdot y=y\cdot x$, $x\cdot0=0$, $x\cdot1=x$, $\big(x\cdot(y\cdot z)\big)\cdot z=\big(x\cdot(y\cdot z)\big)\cdot1$,
\item $(x+1)+1=x$,
\item $x\cdot y=x$ implies $(x+1)\cdot(y+1)=y+1$,
\item $x\cdot(x+1)=0$,
\item $\big((x+1)\cdot(y+1)+1\big)\cdot z=\big((x\cdot z+1)\cdot(y\cdot z+1)+1\big)\cdot1$,
\item $x\in A\cdot B$ if and only if $x\cdot y=x$ for all $y\in A\cup B$ and if $x\cdot z=x$ and $z\cdot y=z$ for all $y\in A\cup B$ together imply $z=x$.
\end{enumerate}	
\end{definition}

How a dual of a Boolean poset can be is derived from a given Boolean poset is illuminated by the next result.

\begin{theorem}
Let $\mathbf P=(P,\le,{}',0,1)$ be a Boolean poset satisfying the {\rm ACC} and the {\rm DCC} and define binary operators $+$ and $\cdot$ on $P$ by
\begin{align*}
     a+b & :=\Min U\big(\Max L(a',b),\Max L(a,b')\big), \\
a\cdot b & :=\Max L(a,b)	
\end{align*}
for all $a,b\in P$ and
\begin{align*}
     A+B & :=\Min U\big(\Max L(A',B),\Max L(A,B')\big), \\
A\cdot B & :=\Max L(A,B)	
\end{align*}
for all $A,B\subseteq P$. Then $\mathbb D(\mathbf P):=(P,+,\cdot,0,1)$ is a dual of a Boolean poset.
\end{theorem}

\begin{proof}
Let $a,b,c\in P$ and $A,B\subseteq P$.
\begin{enumerate}[(i)]
\item We have
\begin{align*}
a\cdot a & =\Max L(a,a)=a, \\
a\cdot b & =\Max L(a,b)=\Max L(b,a)=b\cdot a, \\
 a\cdot0 & =\Max L(a,0)=0, \\
 a\cdot1 & =\Max L(a,1)=a.	
\end{align*}
Moreover, $x\le y\le c$ for all $x\in \Max L\big(a,\Max L(b,c)\big)$ and all $y\in\Max L(b,c)$. Since $\Max L(b,c)\ne\emptyset$ according to the remark after Lemma~\ref{lem2} we conclude $x\le c$ for all $x\in \Max L\big(a,\Max L(b,c)\big)$ and hence
\begin{align*}
\big(a\cdot(b\cdot c)\big)\cdot c & =\Max L\Big(\Max L\big(a,\Max L(b,c)\big),c\Big)= \\
                                  & =\Max L\Big(\Max L\big(a,\Max L(b,c)\big)\Big)= \\
                                  & =\Max L\Big(\Max L\big(a,\Max L(b,c)\big),1\Big)=\big(a\cdot(b\cdot c)\big)\cdot1
\end{align*}
since $\Max L\big(a,\Max L(b,c)\big)\ne\emptyset$ due to the remark after Lemma~\ref{lem2}.
\item We have
\[
a+1=\Min U\big(L(a',1),L(a,1')\big)=\Min U(a',0)=a'
\]
and hence $(a+1)+1=(a')'=a$.
\item The following are equivalent: $a\cdot b=a$, $\Max L(a,b)=a$, $a\le b$. Now every of the following statements implies the next one: $a\cdot b=a$, $a\le b$, $b'\le a'$, $b'\cdot a'=b'$, $a'\cdot b'=b'$, $(a+1)\cdot(b+1)=b+1$.
\item We have $a\cdot(a+1)=\Max L(a,a')=0$.
\item We have $\Min U(A,B)=\big(\Max L(A',B')\big)'=(A'\cdot B')'$. Using distributivity of $\mathbf P$ we conclude
\begin{align*}
\big((a+1)\cdot(b+1)+1\big)\cdot c & =(a'\cdot b')'\cdot c=\Max L\big(\Min U(a,b),c\big)= \\
                                   & =\Max L\big(U(a,b),c\big)=\Max LU\big(L(a,c),L(b,c)\big)= \\
                                   & =\Max L\Big(U\big(L(a,c),L(b,c)\big),1\Big)= \\
                                   & =\Max L\Big(\Min U\big(\Max L(a,c),\Max L(b,c)\big),1\Big)= \\
                                   & =\big((a\cdot c)'\cdot(b\cdot c)'\big)'\cdot1.
\end{align*}
\item The following are equivalent: $a\in A\cdot B$, $a\in\Max L(A,B)$, $a\in L(A,B)$ and if $a\le b\in L(A,B)$ then $b=a$.
\end{enumerate}	
\end{proof}

For finite duals of Boolean posets we can prove the following.

\begin{theorem}
Let $\mathbf D=(P,+,\cdot,0,1)$ be a finite dual of a Boolean poset. Define a binary relation $\le$ on $P$ by $x\le y$ if $x\cdot y=x$, and a unary operation $'$ on $P$ by $x':=x+1$ {\rm(}$x,y\in P${\rm)}. Then $\mathbb B(\mathbf D):=(P,\le,{}',0,1)$ is a Boolean poset.
\end{theorem}

\begin{proof}
Let $a,b,c\in P$ and $A,B\subseteq P$. We consider the conditions of Definition~\ref{def3}. Reflexivity of $\le$ follows from (i). If $a\le b$ and $b\le a$ then $a\cdot b=a$ and $b\cdot a=b$ and hence $a=a\cdot b=b\cdot a=b$ according to (i). If $a\le b$ and $b\le c$ then $a\cdot b=a$ and $b\cdot c=b$ and hence $a\cdot c=(a\cdot b)\cdot c=\big(a\cdot(b\cdot c)\big)\cdot c=\big(a\cdot(b\cdot c)\big)\cdot1=(a\cdot b)\cdot1=a\cdot1=a$ by (i) showing $a\le c$. This proves $(P,\le)$ to be a poset. Moreover, $(a')'=(a+1)+1=a$ by (ii). If $a\le b$ then $a\cdot b=a$ and hence $b'\cdot a'=a'\cdot b'=(a+1)\cdot(b+1)=b+1=b'$ by (i) and (iii) showing $b'\le a'$. Hence $'$ is an antitone involution on $(P,\le)$. Because of (vi) we have $A\cdot B=\Max L(A,B)$. Now $\Max L(a,a')=a\cdot a'=a\cdot(a+1)=0$ by (iv) which implies $L(a,a')=0$. Since $'$ is an antitone involution on $(P,\le)$, we conclude $U(a,a')=1$. This shows that $\mathbb B(\mathbf D)$ is a complemented poset. Since $\Max L(A,B)=A\cdot B$ and $'$ is an antitone involution on $(P,\le)$ we see that $\Min U(A,B)=\big(\Max L(A',B')\big)'=(A'\cdot B')'$. Finally, using (v) and Lemma~\ref{lem2} we conclude
\begin{align*}
\Max L\big(U(a,b),c\big) & =\Max L\big(\Min U(a,b),c\big)=\Max L\big((a'\cdot b')',c\big)=(a'\cdot b')'\cdot c= \\
		                & =\big((a\cdot c)'(b\cdot c)'\big)'\cdot1=\Max L\Big(\big((a\cdot c)'(b\cdot c)'\big)',1\Big)= \\
		                & =\Max L\Big(\big((a\cdot c)'(b\cdot c)'\big)'\Big)=\Max L\big(\Min U(a\cdot c,b\cdot c)\big)= \\
		                & =\Max LU(a\cdot c,b\cdot c)=\Max LU\big(\Max L(a,c),\Max L(b,c)\big)= \\
		                & =\Max LU\big(L(a,c),L(b,c)\big).
\end{align*}
Since $P$ is finite, $(P,\le)$ satisfies the ACC and the DCC. Applying Lemma~\ref{lem2} we conclude $L\big(U(a,b),c\big)=LU\big(L(a,c),L(b,c)\big)$.
\end{proof}

The next theorem shows that the correspondence between Boolean posets and their duals is nearly one-to-one.

\begin{theorem}
\
\begin{enumerate}[{\rm(i)}]
\item If $\mathbf P=(P,\le,{}',0,1)$ is a finite Boolean poset then $\mathbb B\big(\mathbb D(\mathbf P)\big)=\mathbf P$.
\item If $\mathbf D=(P,+,\cdot,0,1)$ is a finite dual of a Boolean poset and $\mathbb D\big(\mathbb B(\mathbf D)\big)=(P,\oplus,\odot,0,1)$ then $(P,\odot,0,1)=(P,\cdot,0,1)$.	
\end{enumerate}
\end{theorem}

\begin{proof}
\
\begin{enumerate}[(i)]
\item Let $\mathbf P=(P,\le,{}',0,1)$ be a finite Boolean poset, $\mathbb D(\mathbf P)=(P,+,\cdot,0,1)$, $\mathbb B\big(\mathbb D(\mathbf P)\big)=(P,\sqsubseteq,{}^*,0,1)$ and $a,b\in P$. Then the following are equivalent: $a\sqsubseteq b$, $a\cdot b=a$, $\Max L(a,b)=a$, $a\le b$. Moreover,
\[
a^*=a+1=\Min U\big(L(a',1),L(a,1)\big)=\Min U(a',0)=a'.
\]
\item If $\mathbf D=(P,+,\cdot,0,1)$ is a finite dual of a Boolean poset, $\mathbb B(\mathbf D)=(P,\le,{}',0,1)$, $\mathbb D\big(\mathbb B(\mathbf D)\big)=(P,\oplus,\odot,0,1)$ and $A,B\subseteq P$ then $A\odot B=\Max L(A,B)=A\cdot B$ according to (vi) of Definition~\ref{def3}.
\end{enumerate}
\end{proof}








Authors' addresses:

Ivan Chajda \\
Palack\'y University Olomouc \\
Faculty of Science \\
Department of Algebra and Geometry \\
17.\ listopadu 12 \\
771 46 Olomouc \\
Czech Republic \\
ivan.chajda@upol.cz

Miroslav Kola\v r\'ik \\
Palack\'y University Olomouc \\
Faculty of Science \\
Department of Computer Science \\
17.\ listopadu 12 \\
771 46 Olomouc \\
Czech Republic \\
miroslav.kolarik@upol.cz

Helmut L\"anger \\
TU Wien \\
Faculty of Mathematics and Geoinformation \\
Institute of Discrete Mathematics and Geometry \\
Wiedner Hauptstra\ss e 8-10 \\
1040 Vienna \\
Austria, and \\
Palack\'y University Olomouc \\
Faculty of Science \\
Department of Algebra and Geometry \\
17.\ listopadu 12 \\
771 46 Olomouc \\
Czech Republic \\
helmut.laenger@tuwien.ac.at
\end{document}